\documentclass{amsart}

\usepackage[utf8]{inputenc}
\usepackage{amsmath,amsfonts,amssymb,amsthm}
\usepackage{color}

\theoremstyle{plain}
\newtheorem{claim}{Claim}
\newtheorem{prop}[claim]{Proposition}
\newtheorem{thm}[claim]{Theorem}
\newtheorem{prob}[claim]{Problem}

\newcommand{\eps}{\varepsilon}

\newcommand{\congg}{\cong_{(1,1)}}
\newcommand{\conggg}{\cong_{(3)}}
\newcommand{\Q}{\mathcal{Q}}

\newcommand{\N}{\mathbb{N}}
\newcommand{\X}{\mathcal{X}}

\newcommand{\choq}{\mathbb{K}_{\mathrm{Choq}}}
\newcommand{\perm}{\cong_{\mathrm{perm}}}
\newcommand{\affine}{\approx_a}
\newcommand{\zero}{\vec{0}}
\newcommand{\Aa}{\mathcal{A}}
\newcommand{\Cc}{\mathcal{C}}
\newcommand{\dd}{d}
\newcommand{\Xx}{\mathbb{X}}

\title[Completeness of the homeomorphism relation\ldots]{Completeness of the homeomorphism relation of locally connected continua}
\author{Tomasz Cie\'{s}la}
\address{Institute of Mathematics, Faculty of Mathematics, Informatics and Mechanics, University of Warsaw}
\email{t.ciesla@mimuw.edu.pl}
\begin{document}

\maketitle

\begin{abstract}
In this paper we prove that the homeomorphism relation of locally connected continua is a complete orbit equivalence relation. This answers a question posed by Chang and Gao in \cite{changgao}.
\end{abstract}

\section{Introduction}

A Borel action $a$ of a Polish group $G$ on a standard Borel space $X$ determines an equivalence relation $E_a$ given by $x E_a y \iff \exists g \in G \ g x=y$. In other words, $xE_ay$ if and only if $x$ and $y$ are in the same orbit of the action $a$. Such relations are called \textit{orbit equivalence relations}. Note that every orbit equivalence relation is analytic, i.e. the set $\{ (x,y) \in X \times X \colon xE_ay\}$ is an analytic subset of the product $X \times X$.

Given two orbit equivalence relations $E$ and $F$ on standard Borel spaces $X$ and $Y$, respectively, we say that a Borel map $f \colon X \to Y$ \textit{reduces $E$ to $F$} if and only if for every $x,y\in X$
$$x E y \iff f(x) F f(y).$$
If this is the case we say that $E$ is \textit{Borel reducible} to $F$.

If $E$ is Borel reducible to $F$ and $F$ is Borel reducible to $E$ then we say that $E$ and $F$ are \textit{Borel bireducible}. Roughly speaking, this means that $E$ and $F$ are of the same complexity.

If $E$ is an orbit equivalence relation such that every orbit equivalence relation $F$ is reducible to $E$ then we say that $E$ is \textit{complete} (or \textit{universal}) orbit equivalence relation. Complete orbit equivalence relations are, in a sense, the most complex objects in the class of orbit equivalence relations. It is known that complete orbit equivalence relations exist, on abstract grounds. This follows from the existence of universal Polish groups and the Mackey-Hjorth theorem \cite{gao}[Theorem 3.5.2] on extensions of actions of Polish groups. On the other hand, the first natural example of a complete orbit equivalence relation is the isometry relation of Polish metric spaces as proved by Gao and Kechris \cite{gaokechris} and Clemens \cite{clemens}. Interestingly enough, recently Melleray \cite{melleray} proved that there exists a Polish metric space whose group of isometries with its natural action on the space induces a complete orbit equivalence relation.

In recent years there has been a considerable amount of research on the classification program of separable C*-algebras from a descriptive set-theoretical point of view. This began with the work of Farah, Toms and Tornquist \cite{farah} and later Elliott, Farah, Paulsen, Rosendal, Toms and Tornquist \cite{elliot} and led to the question of the complexity of the isometry relation of separable C*-algebras. This problem has been solved by Sabok \cite{sabok} who showed that the isometry relation of separable C*-algebras is a complete orbit equivalence relation. Soon thereafter, Zielinski \cite{zielinski}, using Sabok’s result, solved the long-standing problem whether the homeomorphism relation of compact metric spaces is a complete orbit equivalence relation. The latter result was subsequently improved by Chang and Gao \cite{changgao} who showed that the homeomorphism relation of continua (connected compact metric spaces) is also a complete orbit equivalence relation.

These results lead to a number of open questions.

\begin{prob}[Zielinski, \cite{zielinski}]
Is the homeomorphism relation of homogeneous compact metric spaces a complete orbit equivalence relation?
\end{prob}

This problem seems to be very difficult as there are not so many known ways to construct homogeneous spaces.

\begin{prob}[Chang, Gao, \cite{changgao}]\label{prob-changgao}
Is the homeomorphism relation of locally connected continua a complete orbit equivalence relation?
\end{prob}

%This problem is a natural generalization of the main result of \cite{changgao} stating that the homeomorphism relation of continua is complete orbit equivalence relation. 

In this paper we prove that the answer to Problem \ref{prob-changgao} is affirmative.

%The main result of this paper is an affirmative answer to the second question:

%The main result of this paper is the following
\begin{thm} \label{main-result}
The homeomorphism relation of locally connected continua is a complete orbit equivalence relation.
\end{thm}

%Note that this theorem provides an answer to an analogous seemingly more general question regarding locally path-connected continua. However, every locally connected continuum is locally path-connected (cf. \cite{hocking}). Thus the class of locally path-connected continua coincides with the class of locally connected continua. 

Recall that every compact metric space embeds in the Hilbert cube $\Q=[0,1]^{\N}$ and the family $K(\Q)$ of all compact subsets of $\Q$ has a natural Borel structure steming from the Vietoris topology. Kuratowski proved that the set of locally connected subcontinua of $\Q$ is an $F_{\sigma\delta}$ subset of $K(\Q)$ \cite{kuratowski}. This gives a Borel structure on the collection of locally connected continua. It is worth noting that local connectedness and local path-connectedness are equivalent in the class of continua \cite{hocking}. Such continua are also called Peano continua, as they are continuous images of the interval.
%The homeomorphism relation of locally connected continua 
%is then Borel bi-reducible with 
%arises from the action of the group of autohomeomorphisms of $\Q$ in the following way: $X \sim Y$ if and only if there exists a homeomorphism $f \colon \Q \to \Q$ with $f[X]=Y$.
The group of autohomeomorphisms of $\Q$ acts on the collection of locally connected continua in an obvious way. The homeomorphism relation of locally connected continua is an orbit equivalence relation defined in the following way: $X \sim Y$ if and only if there exists a homeomorphism $f \colon \Q \to \Q$ with $f[X]=Y$.

Section \ref{coding-spaces} is devoted to a description of coding spaces. They are used in the last section in which we prove Theorem \ref{main-result}.

\section{The coding spaces} \label{coding-spaces}

Let $\dd$ be a metric on $\Q$ given by the formula $\dd((x_n)_{n\in\N}, (y_n)_{n\in\N})=\sum_{n\in\N}\frac{|x_n-y_n|}{2^n}$. Let $\dd'$ be a metric on $\Q \times \Q$ given by $\dd'((x,y),(z,t))=\dd(x,z)+\dd(y,t)$. We also denote $\zero=(0,0,0,\ldots)\in\Q$ and $e_i=(\underbrace{0,0,\ldots,0}_{i \text{ times }},1,0,0,\ldots)\in\Q$.

In this section we consider locally connected continua $X,Y \subset \Q$ and non-empty families (finite or countably infinite) $\Aa = \{A_n \colon n<|\Aa| \}, \Cc=\{C_n \colon n<|\Cc|\}$ of non-empty closed convex subsets of $\Q$ such that $\bigcup \Aa$ is a closed subset of $X$ and $\bigcup \Cc$ is a closed subset of $Y$.

%In this section $X\subset \Q$ denotes a compact subset of $\Q$ and $\Aa = \{A_n \colon n<|\Aa| \}$ is a non-empty family (finite or countably infinite) of non-empty closed convex subsets of $X$ whose union is a closed subset of $X$ as well. 
%with $\bigcup \Aa=X$.

For every $A \in \Aa$ let $a_0^A, a_1^A, \ldots$ be an enumeration of a dense subset of $A$ in which every element appears infinitely many times. Define $b_k^A = (a_k^A, e_{\langle n,k\rangle}) \in X \times \Q$, where $n$ is such that $A=A_n$ and $\langle \cdot,\cdot\rangle$ is a bijection between $\N \times \N$ and $\N$. 

We define 
$$X'=X \times \{\zero\} \cup \{ b_k^A \colon A \in \Aa, \ k \in \N\}.$$ 
The idea is that for every $A\in\Aa$ we introduce a set of new isolated points whose boundary is precisely the set $A$.
%$X'$ consists of a homeomorphic copy $X \times \{\zero\}$ of $X$ and for every $A \in \Aa$ the set $\{b_0^A, b_1^A, b_2^A, \ldots\}$ consists of isolated points whose set of accumulation points is $A\times \{\zero\}$.
Note that $X'$ is a compact space. 
%Indeed, if $\cal V$ is an open cover of $X'$ then, due to compactness of $X \times \{\zero\}$, there is finite family $\mathcal V' \subset \mathcal V$ such that $\bigcup \mathcal V' \supset X \times \{\zero\}$. But 

Similarly, for every $C\in\Cc$ we consider an enumeration with infinite repetitions $c_0^C, c_1^C, \ldots$ of a dense subset of $C$ and we define $d_k^C=(c_k^C, e_{\langle n,k\rangle})\in Y \times \Q$, where $n$ is such that $C=C_n$. We define $$Y'=Y \times \{\zero\} \cup \{ d_k^A \colon A \in \Aa, \ k \in \N\}.$$ 

%We also consider another compact convex set $Y\subset \Q$ and its analogues $\Cc$, $C_n$, $c_k^C$, $d_k^C$, $Y'$, $Y_k^C$ and $Y''$ of $\Aa$, $A_n$, $a_k^A$, $b_k^A$, $X'$, $X_k^A$ and $X''$, respectively.

A standard back-and-forth construction yields the following

\begin{prop}\label{extends}
If $f \colon X \to Y$ is a homeomorphism such that $\{f[A] \colon A \in \Aa\} = \Cc$ then there is a homeomorphism $g \colon X' \to Y'$ extending $f$ such that $g[\{b_k^A \colon k \in \N\}] = \{d_k^{f[A]} \colon k \in \N\}$ for every $A\in\Aa$.
\end{prop}

For every $A\in\Aa$ and $k\in\N$ we define $X_k^A = \{t\cdot x + (1-t) b_k^A \colon 0\le t \le 1, \ x \in A \times \{\zero\}\}$, i.e. $X_k^A$ is the cone with base $A \times \{ \zero \}$ and apex $b_k^A$. Further, let 
$$X'' = X \times \{\zero\} \cup \bigcup_{A \in \Aa} \bigcup_{k \in \N} X_k^A.$$
That is, for every $A \in \Aa$ we build a sequence of cones with base $A$ such that the boundary of the set of apexes of these cones is $A$. Moreover, every two cones have no common points lying outside $X \times \{\zero\}$.

Note that $X''$ is a locally connected continuum. Compactness follows from the assumption that all sets $A\in\Aa$ are convex.

%First we prove that $X''$ is compact. Take any open cover $\mathcal V$ of $X''$. Due to compactness of $X \times \{ \zero \}$, there is finite $\mathcal V' \subset \mathcal V$ such that $X \times \{ \zero\} \subset \bigcup \mathcal V'$. Convexity of every $A \in \Aa$ and the fact that $b_k^A$ lies ''above'' $A$ implies that $\mathcal V'$ covers all but finitely many cones $X_k^A$. By compactness of the cones $X_k^A$, the union of the cones which are not covered by $\mathcal V'$ is covered by some finite $\mathcal V'' \subset \mathcal V$. Then $\mathcal V' \cup \mathcal V''$ is a finite cover of $X''$.

It is clear that $X''$ is connected, so $X''$ is a continuum. It is also locally connected. Indeed, local connectedness at points outside $X \times \{ \zero \}$ is clear as these points have arbitrarily small convex neighbourhoods. For the proof of local connectednes at points of the form $(x,\zero)\in X \times \{ \zero \}$ we note that any basic neighbourhood $U \times V \subset X \times \Q$ of $(x,\zero)$ with $U\ni x$ being connected is connected. Indeed, since every set $A\in \Aa$ is convex and the apex of every cone $X_n^A$ lies above $A$ it follows that for every point $(y,s)\in X_n^A$ the segment with endpoints $(y,s), (y,\zero)$ is a subset of $X_n^A$. It follows that the set $U\times V$ has the same property as well. Therefore each point $(y,s)\in U \times V$ lies in the same component of $U \times V$ as $(y,\zero)$. But all points of the form $(y,\zero)$ lie in the same component due to the assumption that $U$ is connected. It follows that $U \times V$ is connected.

%So every two points $(y,s), (z,t) \in U\times V$ can be connected by a path consisting of segment with endpoints $(y,s), (y,\zero)$, a path connecting $(y,\zero)$ with $(z,\zero)$ and a segment with endpoints $(z,\zero), (z,t)$.

%This is because if $(y,t\cdot e_{\langle n,k\rangle})\in X''$ then for every $0<s<t$ point $(y,s\cdot e_{\langle n,k\rangle})$ belongs to $X''$ due to the assumption that $A_n$ is convex and the apex $b_i^{A_n}$ lies above $A_n$.

We use the same notation for $Y$, and for every $C \in \Cc$ and $k\in\N$ let $Y_k^C$ be the cone with base $C \times \{\zero\}$ and apex $d_k^C$ and let $Y''$ be the union of $Y\times \{\zero\}$ and all the cones $Y_k^C$.

\begin{prop}\label{extendss}
If $f \colon X \to Y$ is a homeomorphism such that $\{f[A] \colon A \in \Aa\} = \Cc$ and for every $A\in\Aa$ the restriction $f|_A$ is affine then there is a homeomorphism $h \colon X'' \to Y''$ extending $f$ such that for every $A \in \Aa$ $f[\{b_k^A \colon k \in \N\}] = \{d_k^{f[A]} \colon k \in \N\}$ and for every $A\in\Aa$ and $k\in\N$ the restriction of $h$ to $X_k^A$ is affine.
\end{prop}

\begin{proof}
Let $g\colon X' \to Y'$ be a homeomorphism extending $f$ constructed in the previous proposition. We define $h \colon X'' \to Y''$ by $h|_{X'}=g$, and for every $A \in \Aa$, $x \in A$, $k \in \N$, $0<t<1$
$$h(t\cdot(x,\zero) + (1-t) \cdot b_k^A) = t\cdot g(x,\zero) + (1-t)\cdot g(b_k^A).$$ 

This map is a bijection between compact spaces, so to prove that $h$ is a homeomorphism we only have to show that $h$ is continuous. It is also clear from the definition of $h$ that $h|_{X_k^A}$ is affine for every $A \in \Aa$ and $k \in \N$.

%Let $\dd$ be a metric on $\Q$ given by the formula $\dd((x_n)_{n\in\N}, (y_n)_{n\in\N})=\sum_{n\in\N}\frac{|x_n-y_n|}{2^n}$. Let $\dd'$ be a metric on $\Q \times \Q$ given by $\dd'((x,y),(z,t))=\dd(x,z)+\dd(y,t)$. 

Let $t_j(x_j,\zero)+(1-t_j)b_{k_j}^{A_{n_j}}$, where $k_j, n_j \in \N, x_j \in A_{n_j}, t_j\in[0,1]$ be a sequence of elements of the domain of $h$ converging to some $t(x,0)+(1-t)b_k^{A_n}$ (where $x\in A_n, t\in[0,1], k,n\in\N$). If $t<1$ then for sufficiently large $j$ we have $k_j=k$, $n_j=n$ and also $\displaystyle\lim_{j\to\infty} x_j=x$, $\displaystyle \lim_{j\to\infty} t_j=t$. Therefore

\begin{multline*}
h(t_j(x_j,\zero)+(1-t_j)b_{k_j}^{A_{n_j}}) = t_jg(x_j,\zero)+(1-t_j)g(b_{k_j}^{A_{n_j}}) \\
\xrightarrow{j\to\infty} tg(x,\zero)+(1-t)g(b_k^{A_n}) = h(t(x,\zero)+(1-t)b_k^{A_n}).
\end{multline*}

If $t=1$ then $t_j(x_j,\zero)+(1-t_j)b_{k_j}^{A_{n_j}} \xrightarrow{j\to\infty} (x,\zero)$. Fix $\eps>0$. Pick an integer $N$ so large that if $\langle n,k\rangle > N$ and $g(b_k^{A_n})=d_l^{C_m}$ then $\dd(c_l^{C_m},f(a_k^{A_n}))<\eps/4$ and $2^{-\langle m,l\rangle}<\eps/4$. Pick an integer $N'$ such that whenever $j>N'$ then $(1-t_j)2^{-\langle n_j,k_j\rangle} < \eps2^{-N-2}$ and $\dd(f(x),t_jf(x_j)+(1-t_j)f(a_{k_j}^{A_{n_j}}))<\eps/4$.

We have $h(x,\zero)=g(x,\zero)=(f(x),\zero)$ and 

\begin{align*}
h(t_j(x_j,\zero)+(1-t_j)b_{k_j}^{A_{n_j}}) 
& = t_jg(x_j,\zero)+(1-t_j)g(b_{k_j}^{A_{n_j}}) \\
& = (t_jf(x_j)+(1-t_j)f(a_{k_j}^{A_{n_j}}),\zero) \\
& \quad + (1-t_j)(c_{l_j}^{C_{m_j}}-f(a_{k_j}^{A_{n_j}}), e_{\langle m_j, l_j \rangle}),
\end{align*}
therefore, if $j>N'$ then

\begin{align*}
\dd'(h(x,\zero), h & (t_j(x_j,\zero)+(1-t_j)b_{k_j}^{A_{n_j}})) \\
& \le \dd(f(x),t_jf(x_j)+(1-t_j)f(a_{k_j}^{A_{n_j}})) \\
& \quad + (1-t_j)\dd(c_{l_j}^{C_{m_j}},f(a_{k_j}^{A_{n_j}}))+(1-t_j)2^{-\langle m_j, l_j\rangle} \\
& \le \eps/4 + (1-t_j)\dd(c_{l_j}^{C_{m_j}},f(a_{k_j}^{A_{n_j}}))+(1-t_j)2^{-\langle m_j, l_j\rangle} = (*).
\end{align*}

If $1-t_j\ge \eps/4$ then $\eps 2^{-\langle n_j,k_j\rangle-2} \le (1-t_j)2^{-\langle n_j,k_j\rangle} < \eps2^{-N-2}$, i.e. $\langle n_j,k_j\rangle>N$, so $\dd(c_l^{C_m},f(a_k^{A_n}))<\eps/4$ and it follows that $(*) < \eps/4 + (1-t_j)\eps/4 + \eps2^{-N-2} \le \eps$.

Otherwise $1-t_j<\eps/4$ and $(*) \le \eps/4 + \eps/4\cdot \sup \dd' + \eps/4 = \eps$. 

It follows that $h$ is a continuous function. This finishes the proof. 

\end{proof}

Now, for every $k\in \N$ and $A\in\Aa$ let $\hat b_k^A$ and $\tilde b_k^A$ be two distinct points in $\Q \times \Q$ such that $\dd'(b_k^A, \hat b_k^A)=\dd'(b_k^A, \tilde b_k^A)=\frac 1{2+\langle n,k \rangle}$, where $n$ is such that $A_n=A$, and $\hat b_k^A - (\zero,e_{\langle n,k\rangle}), \tilde b_k^A - (\zero,e_{\langle n,k \rangle}) \in \Q \times \{\zero\}$. We denote $\hat I_k^A = \{ tb_k^A +(1-t)\hat b_k^A \colon 0\le t \le 1\}$ and $\tilde I_k^A = \{ tb_k^A +(1-t)\tilde b_k^A \colon 0\le t \le 1\}$. We define

$$T(X,\Aa)=X'' \cup \bigcup_{A\in\Aa} \bigcup_{k \in \N} \hat I_k^A \cup \tilde I_k^A.$$
In other words, we consider the space $X''$ and for every $k \in \N$ and $A \in \Aa$ we attach two short segments $\hat I_k^A$ and $\tilde I_k^A$ to the apex of $X_k^A$. The key property of points $b_k^A$ in $T(X,\Aa)$ is that $T(X,\Aa)\setminus \{b_k^A\}$ consists of three connected components.

Note that $T(X,\Aa)$ is a locally connected continuum. Compactness of $T(X,\Aa)$ is proved similarly as of $X''$. It is clear that $T(X,\Aa)$ is connected. Local connectedness of $T(X,\Aa)$ easily follows from local connectedness of $X''$.

We define similarly $\hat d_k^C$ and $\tilde d_k^C$ as points at the distance $1/(2+\langle n,k \rangle)$ from $d_k^C$, where $C=C_n$, we denote the segment with endpoints $d_k^C, \hat d_k^C$ as $\hat I_k^C$ and the segment with endpoints $d_k^C, \tilde d_k^C$ as $\tilde I_k^C$. We define $T(Y,\Cc)$ as the union of $Y''$ and all the segments $\hat I_k^C$, $\tilde I_k^C$.

\begin{prop}\label{extendsss}
If $f \colon X \to Y$ is a homeomorphism such that $\{f[A] \colon A \in \Aa\} = \Cc$ and for every $A\in \Aa$ the restriction $f|_A$ is affine then there is a homeomorphism $h' \colon T(X,\Aa) \to T(Y,\Cc)$ extending $f$ such that for every $A \in \Aa$ $h'[\{b_k^A \colon k \in \N\}] = \{d_k^{f[A]} \colon k \in \N\}$ and for every $A\in\Aa$ and $k\in\N$ the restrictions of $h'$ to $X_k^A$, $\hat I_k^A$, and $\tilde I_k^A$ are affine.
\end{prop}

\begin{proof}
We simply extend the homeomorphism $h \colon X'' \to Y''$ constructed in the previous proposition to $h' \colon T(X,\Aa) \to T(Y,\Cc)$ by the formula $h'(tb_k^A + (1-t)\hat b_k^A) = td_l^C + (1-t)\hat d_l^C$ and $h'(tb_k^A + (1-t)\tilde b_k^A) = td_l^C + (1-t)\tilde d_l^C$, where $h(b_k^A)=d_l^C$.
\end{proof}

We will also need a variant of the space $T(X,\Aa)$. Consider a convex closed set $B \notin \Aa$. Consider the space $T(X,\Aa \cup \{B\})$. For every $k\in\N$ let $\check b_k^B$ be a point distinct from $\hat b_k^B, \tilde b_k^B$ with $\dd'(b_k^B, \check b_k^B)=\dd'(b_k^B, \hat b_k^B)$ and $\check b_k^B - b_k^B + a_k^B \in \Q \times \{\zero\}$. Denote the closed segment with endpoints $b_k^B$, $\check b_k^B$ by $\check I_k^B$. We define 
$$T'(X,B,\Aa) = T(X,\Aa \cup \{B\}) \cup \bigcup_{k\in\N}\check I_k^B,$$ 
that is, we attach an extra segment to the apex of every cone with base $B$, so removing the apex results in four connected components instead of three.

%In the space $T'(X,B,\Aa)$ the apexes of cones with base $B$ are distinguished from the apexes of the other cones by the property that removing it yields four connected components instead of just three.

Clearly $T'(X,B,\Aa)$ is a locally connected continuum.

\begin{prop}\label{extendssss}
If $f \colon X \to Y$ is a homeomorphism such that $\{f[A] \colon A \in \Aa\} = \Cc$ and for every $A\in\Aa$ the restriction $f|_A$ is affine, and $B \subset X$, $D \subset Y$ are closed convex sets such that $f[B]=D$ and $f|_B$ is affine, then there is a homeomorphism $h'' \colon T'(X,B,\Aa) \to T'(Y,D,\Cc)$ extending $f$.
\end{prop}

\begin{proof}
Using Proposition \ref{extendsss} we get a homeomorphism $h' \colon T(X,\Aa \cup \{B\}) \to T(Y,\Cc \cup \{D\})$. We extend it by putting $h''(tb_k^B+(1-t)\check b_k^B)=td_l^D+(1-t)\check d_l^D$, where $h'(b_k^B)=d_l^D$. Then $h''$ clearly is a homeomorphism.
\end{proof}

\section{Homeomorphism relation of locally connected continua is complete}

In this section we will prove the main result. 

Recall that the space $\choq$ of metrizable Choquet simplices is a Borel subset of $K(\Q)$ and that the relation $\affine$ of affine homeomorphism on $\choq$ is complete (this is due to Sabok \cite{sabok}).

Consider a relation $\conggg$ introduced by Zielinski in \cite{zielinski} defined on the space $\{(X,R) \in K(\Q) \times K(Q^3) \colon R \subset X^3\}$, where $(X,R) \conggg (Y,S)$ if and only if there is a homeomorphism $f \colon X \to Y$ with $f^3[R]=S$. Here, $f^3$ means $f^3(x,y,z)=(f(x),f(y),f(z))$. Consider the map $\Gamma \colon \choq \to K(\Q^3)$ given by the formula
$$\Gamma(X) = \{ (x,y,z) \in X^3 \colon \frac 12 x + \frac 12 y = z \}.$$

The following is \cite[Proposition 2]{zielinski}
\begin{prop}\label{affine-conggg}
For every $X,Y\in\choq$ the following equivalence holds: $X \affine Y \iff (X, \Gamma(X)) \conggg (Y,\Gamma(Y))$.
\end{prop}

Note that $\Gamma(X)$ is convex for every Choquet simplex $X$.

%Since every Choquet simplex is convex, its image under $\Gamma$ is convex as well. Indeed, let $X$ be a Choquet simplex, $(x,y,z), (x',y',z') \in \Gamma(X)$ and $t \in [0,1]$. Then $\frac 12 x + \frac 12 y = z$ and $\frac 12 x' + \frac 12 y' = z'$ which implies that
%$$\frac 12(tx+(1-t)x') + \frac 12(ty+(1-t)y') = tz+(1-t)z'.$$
%Since $X$ is convex, we have $tx+(1-t)x', ty+(1-t)y', tz+(1-t)z'\in X$, therefore $$t(x,y,z)+(1-t)(x',y',z') = (tz+(1-t)x', ty+(1-t)y', tz+(1-t)z') \in \Gamma(X).$$

We recall another relation from \cite{zielinski}. Let $\perm$ be defined on $K(\Q)^\N$ in the following way: $(A_1, A_2, \ldots) \perm (B_1, B_2, \ldots)$ if and only if there exists a homeomorphism $h \colon \Q \to \Q$ and a permutation $\sigma$ of $\N$ such that $h(A_n)= B_{\sigma(n)}$ for any $n$. 

For a Choquet simplex $X$ consider the space $\tilde{X}=T(X,\{X\})$ and write $b_k$ instead of $b_k^X$. Define for any $n\in \N$, 
\begin{align*}
B_n & = \{b_n\} \times \tilde{X}^2 \\
C_n & = \tilde{X} \times \{b_n\} \times \tilde{X} \\
D_n & = \tilde{X}^2 \times  \{b_n\} \\
E_n & = B_n \cup C_n \\
F_n & = B_n \cap D_n.
\end{align*}

Let $\Psi \colon \choq \to K(\Q^3)^\N$ be the function 
$$\Psi(X)=(\tilde{X}^3, \Gamma(X), B_1, C_1, D_1, E_1, F_1, B_2, C_2, D_2, E_2, F_2, \ldots).$$

%\color{red} JAK TO SIE MA DO Z-ZBIOROW \color{black}

The proof of the following proposition is similar to the proof of \cite[Proposition 3]{zielinski}.

\begin{prop}\label{affine-perm}
For every $X,Y\in\choq$ the following equivalence holds: $X \affine Y \iff \Psi(X)\perm\Psi(Y)$.
\end{prop} 

\begin{proof}
%Suppose that $X \affine Y$. Then there is an affine homeomorphism $f \colon X \to Y$. Let $g \colon \tilde X \to \tilde Y$ be the homeomorphism extending $f$ constructed in the proof of proposition \ref{extendsss}. Write $\Psi(Y)=(\tilde{Y}^3, \Gamma(Y), H_1, I_1, J_1, K_1, L_1, H_2, I_2, J_2, \ldots)$ and let $d_n\in\tilde{Y}$ be such that $H_n=\{d_n\}\times \tilde Y ^2$. Then $g$ determines permutation $\tau \colon \N \to \N$ given by $f(b_n)=d_{\tau(n)}$. Define a permutation $\sigma \colon \N \to \N$ by $\sigma(0)=0$, $\sigma(1)=1$, and $\sigma(5k+j+2)=5\tau(k)+j+2$ for every $k\in\mathbb\N$, and $j\in\{0,1,\ldots,4\}$.

%We have $g^3[\tilde X^3]=\tilde Y^3$, $g^3[B_n]=g^3[\{b_n\}\times \tilde X^2]=\{d_{\tau(n)}\} \times \tilde Y^2 = H_{\tau(n)}$, and similarly $g^3[C_n]=I_{\tau(n)}$, $g^3[D_n]=J_{\tau(n)}$, $g^3[E_n]=K_{\tau(n)}$, and $g^3[F_n]=L_{\tau(n)}$. Moreover, by proposition \ref{affine-conggg}, $g^3[\Gamma(X)]=\Gamma(Y)$. Therefore $g^3$ and $\sigma$ witness that $\Psi(X)\perm\Psi(Y)$.

It follows from the proof of \cite[Proposition 3]{zielinski} (where instead of using \cite[Proposition 1]{zielinski} we use Proposition \ref{extendsss}) that $X \affine Y \implies \Psi(X) \perm \Psi(Y)$.

For the implication in the other direction, suppose that $\Psi(X) \perm \Psi(Y)$ witnessed by $\sigma \colon \N \to \N$ and $h\colon \tilde X^3 \to \colon \tilde Y^3$. Write 
$$\Psi(Y)=(\tilde{Y}^3, \Gamma(Y), H_1, I_1, J_1, K_1, L_1, H_2, I_2, J_2, \ldots)$$ 
and let $d_n\in\tilde{Y}$ be such that $H_n=\{d_n\}\times \tilde Y ^2$. Again, the proof of \cite[Proposition 3]{zielinski} shows that $h[\tilde X^3]=\tilde Y^3$, $h[\Gamma(X)]=\Gamma(Y)$, and that there is a permutation $\tau \colon \N \to \N$ with $h[B_n]=H_{\tau(n)}$, $h[C_n]=I_{\tau(n)}$, $h[D_n]=J_{\tau(n)}$, $h[E_n]=K_{\tau(n)}$, $h[F_n]=L_{\tau(n)}$. 

%Then, following the proof of \cite[Proposition 3]{zielinski}, notice that $h[\{(b_n,b_n,b_n\}] = h[B_n\cap C_n\cap D_n] = h[B_n] \cap h[C_n] \cap h[D_n] = H_{\tau(n)} \cap I_{\tau(n)} \cap J_{\tau(n)} = \{(d_{\tau(n)},d_{\tau(n)},d_{\tau(n)})\}$, so $h(b_n,b_n,b_n)=(d_{\tau(n)},d_{\tau(n)},d_{\tau(n)})$. The set $\{(b_n,b_n,b_n) \colon n \in \N \}$ is dense in the diagonal of $(X')^3=(X \cup \{b_n \colon n\in\N\})^3$ so $h[\Delta_{(X')^3}]=h[\cl{\{(b_n,b_n,b_n) \colon n \in \N \}}] = \cl{\{(b_n,b_n,b_n) \colon n \in \N \}} = \Delta_{(Y')^3}$. This determines a homeomorphism $g \colon X' \to Y'$ by the formula $g(x)=y \iff h(x,x,x)=(y,y,y)$.

Denoting $X'=X \cup \{b_n \colon n\in\N\}$ and $Y'=Y \cup \{ d_n \colon n \in \N\}$ the proof of \cite[Proposition 3]{zielinski} shows that $(X,\Gamma(X)) \conggg (Y,\Gamma(Y))$. It follows from the Proposition \ref{affine-conggg} that $X\affine Y$.
\end{proof}

\bigskip

Let 
$$\X = \{(x,y) \in \Q^2 \colon \forall m\neq n, y_m=0 \vee y_n=0\}.$$
For every $(A_1, A_2, \ldots)\in K(\Q)^\N$ we define 
$$\Xi(A_1,A_2,\ldots) = \{(x,y) \in \X \colon \forall n \ y_n=0 \vee x \in A_n\}.$$
We identify $\Q$ with $\Q \times \{\zero\}$.

Recall the definition of the relation $\congg$ from \cite{zielinski}. If $A \subset B \subset X$ and $C \subset D \subset Y$ then $(X,B,A) \congg (Y,D,C)$ if and only if there exists a homeomorphism $f \colon X \to Y$ such that $f[A]=C$ and $f[B]=D$. 

The following is \cite[Proposition 4]{zielinski}.

\begin{prop}\label{extendsssss}
Let $\vec{A}=(A_1, A_2, \ldots)\in K(\Q)^\N$ and $\vec{B}=(B_1, B_2, \ldots) \in K(\Q)^\N$. Then 
$\vec{A} \perm \vec{B}$ if and only if $(\X,\Xi(\vec{A}), \Q) \congg (\X, \Xi(\vec{B}), \Q)$. Moreover, if $f \colon X \to Y$ and $\sigma \colon \N \to \N$ witness that $\vec{A} \perm \vec{B}$ then $f \times h_{\sigma^{-1}}|_{\X}$ witnesses that $(\X, \Xi(\vec{A}), \Q) \congg (\X, \Xi(\vec{B}), \Q)$, where $h_{\tau} \colon \Q \to \Q$ is the homeomorphism given by $h_{\tau}(x_1, x_2, \ldots) = (x_{\tau(1)}, x_{\tau(2)}, \ldots)$.
\end{prop}

\bigskip

Identifying $\Q^3$ with $\Q$ in an obvious way we may treat $\Psi(X)$ as an element of $K(\Q)^\N$. Therefore it makes sense to consider $\Xi(\Psi(X))$.
Note that $\Xi(\Psi(X))$ can be written as a union of a countable family of convex closed subsets of $\X$. Indeed, let $\Xx$ denote the family of cones and segments $\tilde X$ is the union of. Also let $[0,1]_n=\{te_n \colon t\in[0,1]\} \subset \Q$. Then $\Xi(\Psi(X))$ is the union of the following family:

%PONIŻEJ TO Q TRZEBA TRAKTOWAĆ JAK Q^3, W KTÓRYM SIEDZĄ TE PRODUKTY S_1 \TIMES S_2 \TIMES S_3

\begin{align*}
\mathcal F_X & = \{\Q \times \{\zero\}\} \cup 
\{(S_1\times S_2 \times S_3) \times [0,1]_0 \colon S_1, S_2, S_3 \in \Xx \} \cup 
\{ \Gamma(X) \times [0,1]_1\} \\
& \quad \cup \bigcup_{n\in\N} \{ (\{b_n\} \times S_1 \times S_2) \times [0,1]_{5n+2} \colon S_1, S_2 \in\Xx\}  \\
& \quad \cup \bigcup_{n\in\N} \{ (S_1 \times \{b_n\} \times S_2) \times [0,1]_{5n+3} \colon S_1, S_2 \in\Xx\}  \\
& \quad \cup \bigcup_{n\in\N} \{ (S_1 \times S_2 \times \{b_n\}) \times [0,1]_{5n+4} \colon S_1, S_2 \in\Xx\}  \\
& \quad \cup \bigcup_{n\in\N} \{ (\{b_n\} \times S_1 \times S_2) \times [0,1]_{5n+5} \colon S_1, S_2 \in\Xx\}  \\
& \quad \cup \bigcup_{n\in\N} \{ (S_1 \times \{b_n\} \times S_2) \times [0,1]_{5n+5} \colon S_1, S_2 \in\Xx\}  \\
& \quad \cup \bigcup_{n\in\N} \{ (\{b_n\} \times S_1 \times \{b_n\}) \times [0,1]_{5n+6} \colon S_1, S_2 \in\Xx\}.
\end{align*}

%Theorem \ref{main-result} readily follows from the following proposition.

The following proposition gives an explicit reduction of $\affine$ to the homeomorphism relation of locally connected continua. Borelness of this reduction follows from a routine verification. %and will be skipped. 
As a consequence we get Theorem \ref{main-result}.

\begin{prop}
The map $X \mapsto T'(\X, \Q, \mathcal F_X)$ is a reduction of $\affine$ to the homeomorphism relation of locally connected continua.
\end{prop}

\begin{proof}
Let $X,Y\in\choq$. Assume that $f \colon X \to Y$ is an affine homeomorphism. Using Propositions \ref{extendsss} and \ref{extendsssss} we conclude that there is a homeomorphism of $\X$ witnessing that $(\X, \Xi(\Psi(X)), \Q) \congg (\X, \Xi(\Psi(Y)),\Q)$ which is affine on every set from $\mathcal F_X$. It follows then by Proposition \ref{extendssss} that $T'(\X, \Q, \mathcal F_X)$ is homeomorphic to $T'(\X, \Q, \mathcal F_Y)$.

Conversely, assume that $T'(\X, \Q, \mathcal F_X)$ is homeomorphic to $T'(\X, \Q, \mathcal F_Y)$ and let $f$ be a homeomorphism witnessing that. Let $S^X_k \subset T'(\X, \Q, \mathcal F_X)$ be the set of points $x$ such that $T'(\X, \Q, \mathcal F_X) \setminus \{x\}$ consists of exactly $k$ connected components. Note that this property is preserved by homeomorphisms. By construction of $T'(\X, \Q, \mathcal F_X)$, the boundary of the set $S^X_4$ is $\Q$. Therefore, since $f[S_4^X] = S_4^Y$, the image of the boundary of $S_4^X$ is equal to the boundary of $S_4^Y$, i.e. $f[\Q]=\Q$. Analogously, by construction we know that the boundary of the set $S_3^X$ is $\bigcup \mathcal F_X=\Xi(\Psi(X))$. We conclude that $f[\Xi(\Psi(X))]=\Xi(\Psi(Y))$. 
The points $x\in\X \setminus \Xi(\Psi(X))$ are characterized by the following property: the connected component of $x$ in the space $T'(\X, \Q, \mathcal F_X)\setminus\Xi(\Psi(X))$ is disjoint from $S^X_3 \cup S_4^X$. Analogous statement holds for $Y$. It follows that $f[\X \setminus \Xi(\Psi(X))]=\X \setminus \Xi(\Psi(Y))$. Finally, 
\begin{multline*}
f[\X]=f[(\X \setminus \Xi(\Psi(X))) \cup \Xi(\Psi(X))]=f[\X \setminus \Xi(\Psi(X))]\cup f[\Xi(\Psi(X))] \\
= (\X \setminus \Xi(\Psi(Y))) \cup \Xi(\Psi(Y))=\X.
\end{multline*}
Therefore $(\X, \Xi(\Psi(X)),\Q) \congg (\X, \Xi(\Psi(Y)),\Q)$. By Proposition \ref{extendsssss}, $\Psi(X) \perm \Psi(Y)$. In view of Proposition \ref{affine-perm} this is equivalent to $X \affine Y$. This finishes the proof.
\end{proof}

\section*{Acknowledgements}

I am grateful to Marcin Sabok for many valuable discussions on the topic and helpful remarks concerning the early draft of the paper. Also thanks are due to Roman Pol for providing the reference to Kuratowski's result on locally connected continua.

\end{document}